\documentclass[11pt]{amsart}

\usepackage{amsmath, amssymb}
\usepackage{hyperref}

\newcounter{my_enumerate_counter}
\newcommand{\pushcounter}{\setcounter{my_enumerate_counter}{\value{enumi}}}
\newcommand{\popcounter}{\setcounter{enumi}{\value{my_enumerate_counter}}}

\newcommand{\fT}{\mathfrak T} 
\newcommand{\fF}{\mathfrak F} 
\newcommand{\ffF}{\mathfrak {\bar F}}
\newcommand{\cU}{\mathcal U}
\newcommand{\cO}{\mathcal O}
\newcommand{\bfn}{\mathbf n}

\newcommand{\bt}{\mathbf t} 
\newcommand{\bs}{\mathbf s}

\newcommand{\bbD}{{\mathbb D}}

\newcommand{\bbN}{{\mathbb N}}
\newcommand{\bbC}{\mathbb C}
\newcommand{\bbQ}{\mathbb Q}

\newcommand{\cZ}{{\mathcal Z}}

\newcommand{\e}{\varepsilon}
\newtheorem{thm}{Theorem}[section]
\newtheorem{theorem}{Theorem}

\newtheorem{lemma}[thm]{Lemma}

\newtheorem{prop}[thm]{Proposition}
\DeclareMathOperator{\dist}{dist}

\theoremstyle{definition}

\title{Omitting types and AF algebras}

\author[K. Carlson et al.]{KEVIN CARLSON}

\address{Department of Mathematics\\
Indiana University\\ Bloomington, Indiana, 47408}

\email{kdcarlso@indiana.edu}

\author[]{ENOCH CHEUNG}

\address{Department of Mathematics\\ 970 Evans Hall\\ University of California
Berkeley\\ CA 94720}

\email{enochcheung@berkeley.edu}

\author[]{Ilijas Farah}

\address{Department of Mathematics and Statistics\\
York University\\
4700 Keele Street\\
North York, Ontario\\ Canada, M3J
1P3\\
and Matematicki Institut, Kneza Mihaila 35, Belgrade, Serbia}

\urladdr{http://www.math.yorku.ca/$\sim$ifarah}

\email{ifarah@mathstat.yorku.ca}

\author[]{ALEXANDER GERHARDT-BOURKE}

\address{School of Mathematics and Applied Statistics\\
 University
of Wollongong\\
  NSW 2522, Australia}

\email{agb526@uowmail.edu.au}

\author[]{Bradd Hart} 
\address{Dept. of Mathematics and Statistics\\
McMaster University\\ 1280 Main Street\\ West Hamilton, Ontario\\
Canada L8S 4K1}

\email{hartb@mcmaster.ca}

\urladdr{http://www.math.mcmaster.ca/$\sim$bradd/}

\author[]{LEANNE MEZUMAN}

\address{
Department of Mathematics and Statistics\\
University of Guelph\\
 50 Stone Road East\\
Guelph, Ontario\\ Canada N1G 2W1}
 \email{lmezuman@uoguelph.ca}

\author[]{Nigel Sequeira}

\address{Dept. of Mathematics and Statistics\\
McMaster University\\ 1280 Main Street\\ West Hamilton, Ontario\\
Canada L8S 4K1}

\email{sequeinj@mcmaster.ca}

\author[]{ALEXANDER SHERMAN}

\address{Department of Mathematics, 
University  of Maryland, College Park, MD 20742-4015}

\email{asherm92@umd.edu}

\thanks{This research was supported by the Fields--MITACS undergraduate summer research program in July and August 2012.}
 
\date{\today}

\begin{document}

\maketitle

The model theory of metric structures (\cite{BYBHU}) 
was successfully applied to analyze ultrapowers of C*-algebras
in \cite{FaHaSh:Model1} and \cite{FaHaSh:Model2}. 
Since  important classes of separable C*-algebras, such as UHF, AF, or nuclear algebras,  
are not elementary (i.e., not characterized by their theory---see \cite[\S 6.1]{FaHaSh:Model2}), for a moment it seemed that 
 model theoretic methods do not apply  to these classes of C*-algebras. 
 We prove results suggesting that this is not the case.

 Many of the  prominent problems in the modern theory of C*-algebras
 are concerned with the extent of the class of nuclear C*-algebras. 
 We have the bootstrap class problem (see
 \cite[IV3.1.16]{Black:Operator}), the question of whether 
 all nuclear C*-algebras satisfy the Universal 
 Coefficient Theorem, UCT,  (see \cite[\S 2.4]{Ror:Classification}), and the Toms---Winter conjecture (to the effect that the three regularity properties of nuclear C*-algebras discussed in \cite{EllTo:Regularity}
 are equivalent;  see~\cite{Win:Ten}). 
  If one could characterize classes of algebras in question---such as nuclear algebras, 
  algebras with finite nuclear dimension, or algebras with finite decomposition rank---as algebras that 
  omit certain sets of types (see \S\ref{S.type}) then one might use the omitting types theorem (\cite[\S 12]{BYBHU}) to construct such algebras, modulo resolving a number of
  nontrivial technical obstacles.  
This paper is the first, albeit modest, step in this project.

Recall that a unital  C*-algebra is UHF (Uniformly HyperFinite) if 
it is a tensor product of full matrix algebras, $M_n(\bbC)$. 
Non-unital UHF algebras are direct limits of full matrix algebras, 
and in the separable, unital case the two definitions are equivalent. 
A C*-algebra  is AF (Approximately Finite) if it is a direct limit of finite-dimensional C*-algebras. 
These three classes of C*-algebras were the first to be classified -
by work of Glimm, Dixmier and Elliott (building on Bratteli's results), respectively. 
Elliott's classification of separable AF algebras  by the ordered $K_0$ group 
was a prototype for  Elliott's  program for classification of nuclear, simple, separable, unital 
C*-algebras by their K-theoretic invariants (see \cite{Ror:Classification} or \cite{EllTo:Regularity}). 
Types are defined in \S\ref{S.type}. 


\begin{theorem} \label{T.UHF} There are countably many types 
 such that 
a separable C*-algebra $A$ is UHF if and only if it omits all of these types. 
\end{theorem}

\begin{theorem} \label{T.AF} There are countably many types 
 such that 
a separable C*-algebra $A$ is AF if and only if it omits all of these types. 
\end{theorem}

The statement of these two theorems has to be slightly modified in the general, not necessarily separable, case (see \S \ref{S.proofs}).

Recall that two C*-algebras $A$ and $B$ 
are \emph{elementarily equivalent} if they have the same theory. 
In terms of \S\ref{S.type} this corresponds to saying that $A$ and $B$ define the same functional 
on $\ffF_0$. 

\begin{theorem} \label{T.3}
\begin{enumerate}
\item If $A$ and $B$ are unital, separable UHF algebras then they are isomorphic if and only if they are elementarily equivalent. 
\item There are non-unital, separable UHF algebras which are elementarily equivalent
but not isomorphic. 
\item There are unital separable AF algebras which are elementarily equivalent
but not isomorphic. 
\end{enumerate}
\end{theorem}

Although (2) and (3) appear as negative results, they are potentially more interesting than (1).
In  light of  Elliott's classification result and the Effros--Handelman--Shen range of 
invariant result (see \cite{Dav:C*}), the category of ordered dimension groups  
is equivalent to the category of separable AF algebras. 

We note a fact closely related to the above: 
Every separable unital UHF algebra is an atomic model (Theorem~\ref{T.atomic}) but  
non-unital UHF algebras and AF algebras need not be atomic (Proposition~\ref{P.nonatomic}). 

The main motivation behind this work is a desire to initiate a model-theoretic analysis 
of important classes of C*-algebras, such as nuclear algebras, classifiable algebras, 
or algebras with finite decomposition rank or finite nuclear dimension (see e.g., 
\cite{Win:Ten}). 
A logician-friendly introduction to the subject can be found e.g., in 
\cite{Fa:Selected}. 

\subsection*{A remark on terminology} Common English adjectives, such as normal, compact, or stable, 
have special meaning in different areas of mathematics. In the case of compactness 
these meanings cohere to an uncanny degree. However, this phenomenon is quite rare. 
The adjective `stable' is common both in operator algebras  and in model theory - 
model-theoretic stability even turned out to be relevant to questions about 
operator algebras (see~\cite{FaHaSh:Model1}) In C*-algebras the word `stable' has 
more than one different interpretation (see \cite{Black:Operator}  II.6.6., V.3 or II.8.3, page 162)
Nevertheless, in the present paper the word stability 
will always be  interpreted in its C*-algebraic sense, as in `stable relations' (see \cite{Lor:Lifting}
or \cite[p. 166]{Black:Operator}). 
We hope that our  natural choice of terminology will not cause confusion or annoyance.


\section{Model-theoretic types} 

\label{S.type}

Syntax 
 for the  logic for metric structures  was defined in 
 \cite[\S\S 2.3, 2.4]{FaHaSh:Model2}). 
 We shall briefly recall the definitions for C*-algebras.
   \emph{Terms} are *-polynomials in variables $x_n$, $n\in \bbN$. 
 Formulas are defined recursively: 
 \begin{enumerate}
\item  \emph{Atomic formulas} are of the form $\|t\|$ where $t$ is a term. 
  \item If $\phi$ and $\psi$ are formulas and $f\colon [0,\infty)\to [0,\infty)$ then $f(\phi,\psi)$ is 
 a formula. 
 \item If $\phi$ is a formula and $n,k\in \bbN$ then $\inf_{\|x_n\|\leq k} \phi$ and
 $\sup_{\|x_n\|\leq k} \phi$ are formulas.  
 \end{enumerate}
 A formula is $n$-ary if it has at most $n$ free variables. 
 We sometimes write $\phi(x_1,\dots, x_n)$ in order to emphasize that the 
 free variables of $\phi$ are among $x_1, \dots, x_n$. 
 An $n$-ary formula $\phi$ defines a continuous function on $A^n$ for every C*-algebra $A$
 (see \cite[Lemma~2.2]{FaHaSh:Model2}).  
 
 \emph{Types} are defined in \cite[\S 4.3]{FaHaSh:Model2}. 
 We shall need only types over the empty set and we recall the definition now.  
 A \emph{
 condition} is an expression of the form $\phi\leq r$ or $\phi\geq r$, where $\phi$ is a formula 
 and $r\geq 0$ is a real number.  A condition is $n$-ary if the corresponding formula is $n$-ary. 
  An $n$-ary  condition $\phi\geq r$ is \emph{satisfied} in $A$ by $\bar a \in (A_1)^n$ if $\phi(\bar a)^A\geq r$, and similarly for $\phi \leq r$.  Here $A_1$ denotes the unit ball of $A$. 
  (Allowing only the elements of the unit ball to be realizations of a condition is an innocuous 
  restriction and it will simplify some of our considerations.) 
 A set of conditions is a \emph{type}. If the free variables of all conditions in $\bt$ are included in 
 $x_1,\dots, x_n$ we say that $\bt$ is an \emph{$n$-type}. 
An $n$-type $\bt(x_1, \dots, x_n)$ is  
   \emph{consistent} if there exist a C*-algebra $A$ and $\bar a\in (A_1)^n$ such that
   all conditions in $\bt$ are satisfied by $\bar a$.  
 
 This defines a \emph{partial type}. A consistent 
 type that is maximal (as a set of conditions under the inclusion) 
 is also called a \emph{complete type} and sometimes simply a \emph{type}. 
 
 Complete $n$-types form a dual space of a natural Banach space. 
Fix $n\geq 0$ and let $\fF_n$ be the set of all formulas with free variables 
included in $x_1, \dots, x_n$. Then $\fF$ is closed under the addition and multiplication by 
positive reals. Let~$\ffF_n$ be the real vector space generated by $\fF_n$.

For every $n\in \bbN$ we have that $A^n$ and $\ffF_n$ are in duality---albeit  non-linear. 
If   $A$ is a  C*-algebra 
then every $\phi\in \ffF_n$ defines  
a function 
$f_\phi$ on $A^n$ by 
\[
f_\phi(\bar a)=\phi(\bar a)^{A}. 
\]
This function is in general not linear but it is uniformly continuous (see \cite[Lemma~2.2]{FaHaSh:Model2}). 
If $A$ is a C*-algebra then every  $\bar a\in (A_1)^n$ 
 defines a linear functional on $\ffF_n$ by 
\[
g_{A,\bar a} (\phi)=\phi(\bar a)^A. 
\]
These functionals can be used to define a sup-norm on $\ffF_n$ via
\[
\|\phi\|=
\sup_{\bar a\in (A_1)^n} |\phi(\bar a)^A|
=
\sup_{\bar a\in (A_1)^n} g_{A,\bar a}(\phi). 
\]
where the sup is taken over all C*-algebras $A$ and all $a\in (A_1)^n$. 
By the downward L\"owenheim--Skolem theorem (\cite[Proposition~4.7]{FaHaSh:Model2}) 
it suffices to take the sup over all separable C*-algebras. 

Therefore the space of $n$-types can be identified with a closed linear subspace of the 
dual of $\ffF_n$, denoted $\fT_n$. The \emph{logic topology} on the space of types (see \cite[Definition~8.4]{FaHaSh:Model2}) is the weak*-topology  on $\fT_n$
and the \emph{$d$-metric} on types (see the text before Proposition~8.7 in 
\cite{BYBHU}) corresponds to  the dual Banach space  norm on $\fT_n$.

In model theory one usually studies complete types over a complete theory. 
We cannot afford to do that, since the theory of C*-algebras is not complete and the types that 
naturally occur here are not complete. Since the omitting types theorem is stated for complete 
types in complete theories (\cite[\S 12]{BYBHU}) this causes some additional model-theoretic complications (see examples in \cite{BY:Definability} and a simplified version in \cite{Bice:Brief}). 
 
 The idea of providing the space of all formulas with a Banach space structure  can be extended much 
further. In \cite{BY:On} the formulas of the logic of metric structures were equipped with a natural abelian C*-algebra structure and some preliminary results were proved.

 \section{Stable formulas and definability}

Our working assumption is that all models are models of the theory of C*-algebras. 
Nevertheless, some of our results apply to a more abstract context. 
One difference between our treatment and \cite[\S 9]{BYBHU} is that in the latter definability is
considered with respect to a complete theory (or rather, in a fixed model of this theory), 
while we consider definability over a theory that is not complete (the theory of C*-algebras).

A formula $\psi(\bar x)$ is \emph{stable} if for every  
$\e>0$ there exists $\delta>0$ such that
for every  C*-algebra $A$ and every 
$\bar a\in A^n$ with   $|\psi(\bar a)|<\delta$ there exists $\bar b\in A^n$ 
such that $\|a_i-b_i\|<\e$ for all $i<n$ and $\psi(\bar b)=0$. 

The case when $\psi(\bar x)$ is of the form $\|p(\bar x)\|$ for a *-polynomial $p$, 
the above definition corresponds 
to  the notion of \emph{stable relations} (\cite[Definition~14.1.1]{Lor:Lifting}), 
introduced by Blackadar  under the name of  `partially liftable relations' in 
\cite[Remark~2.34]{Black:Shape}.

By $A_1$ we denote the unit ball of a C*-algebra $A$ and consider its $m$-th power (for $m\in \bbN$) 
as a metric space with respect to the max metric. 
Following \cite[Definition~9.16]{BYBHU} we say that a closed subset $D$ of $(A_1)^m$ 
 is \emph{definable} if there exist $m$-ary formulas $\phi_n(\bar x)$, for $n\in \bbN$ 
 such that 
 \begin{equation}
 \lim_n \phi_n(\bar x)=\dist(\bar x, D)
\tag {*}
 \end{equation}
 for all $\bar x\in (A_1)^m$. 
 
 Typically $D$ will be defined as the zero-set of a formula $\psi(\bar x)$. 
 If this zero-set is definable in every C*-algebra by using the same sequence
 of formulas $\phi_n$ and scuh that the rate of 
 convergence in (*) is uniform over all C*-algebras, then we say that it is \emph{uniformly definable}.  
 The following proof is similar to the proof of \cite[Proposition~8.19]{BYBHU}.

\begin{lemma} 
The zero-set of a  formula $\psi$ is uniformly definable if and only if 
the formula is stable. 
\end{lemma} 

\begin{proof} Assume the zero-set $D$ of $\psi$ is uniformly definable. 
Then for every $n$ we can fix $\phi_n$ such that in every C*-algebra $A$ we have 
(using $D^A$ to denote the zero-set of $\psi$ in $A$)
$|\phi_n(\bar b)-\dist(\bar b, D^A)|<1/n$. 
Assume that $\psi$ is not stable. Then for some $\e>0$ and every $n\in \bbN$ there exists
a C*-algebra $A_n$ and $\bar a_n\in (A_n)^m$ such that $\psi(\bar a_n)^{A_n}<1/n$
but $\dist(\bar a_n,D^{A_n})\geq \e$. If $k>2/\e$, then $\phi_k(\bar a_n)^{A_n}\geq \e/2$ for all 
$n$. If $\cU$ is a nonprincipal ultrafilter on $\bbN$ then in the ultraproduct $A=\prod_{\cU} A_n$
the tuple $\bar a$ represented by $(\bar a_n)_{n\in \bbN}$ is such that $\psi(\bar a)^A=0$
but $\phi_l(\bar a)^A\geq \e/2$ for all $l\geq 2/\e$, a contradiction. 

Now assume $\psi$ is stable. By \cite[Proposition~2.10]{BYBHU} 
we can fix a continuous function $\alpha\colon [0,2]\to [0,2]$ (2 is the diameter of the unit ball) 
such that $\alpha(0)=0$, 
$\alpha(2)=2$ and with $D_\psi$ denoting the zero-set of $\psi$ we have  
$\dist(\bar x, D_\psi) \leq\alpha(\psi(\bar x))$ for all $\bar x$. 
Then the formula
\[
\phi(\bar x)=\inf_{\|y\|\leq 1}( \alpha(\psi(\bar y)+\|x-y\|)
\]
is equal to $\dist(\bar x, D_\psi)$ in every C*-algebra $A$. 
\end{proof}

By the above, a subset of a C*-algebra is \emph{definable} 
in the sense of \cite{BYBHU}
iff it is equal to the zero set of a stable formula.

We also note that the  formula 
\[
\rho_p(x)=\|x-x^*\|+\|x^2-x\|
\]
is stable and that its zero set is the set of all projections. 
This is an easy application of continuous functional calculus (see e.g., \cite{Lor:Lifting}). 



%

\begin{lemma} \label{L.alpha} 
For every $n$ there exist  stable formulas $\alpha_n$ and $\alpha_n^u$
 in $n^2$ free variables whose 
zero set in any C*-algebra is the set of matrix units of copies of $M_n(\bbC)$
and the set of matrix units of unital copies of $M_n(\bbC)$, respectively. 
\end{lemma}

\begin{proof} \label{L.Mn} Let $\alpha_n$ be (with $\delta_{kl}$ being Kronecker's delta)
\[
\max_{1\leq i,j,k,l\leq n}\|\delta_{kl}x_{ij}-x_{ik} x_{lj}\|+\|x_{ij}-x_{ji}^*\|+|\|x_{11}\|-1|. 
\]
The zero set of $\alpha_n$ consists of all $n^2$-tuples $a_{ij}$, for $1\leq i,j\leq n$ which 
satisfy the matrix unit equations. 
Stability is a well-known fact (\cite{Black:Shape}, see also \cite{Lor:Lifting}). 

Let $\alpha_n^u$ be $\max(\alpha_n, \|1-\sum_{1\leq i\leq n} x_{ii}|\|)$. 
Stability follows from the fact that projections are defined by a stable formula. 
\end{proof}

\begin{lemma} \label{L.F} For every finite-dimensional C*-algebra $F$ there exist
  stable formulas $\alpha_F$ and $\alpha_F^u$
 in $m$ ($=\dim F$)  free variables whose 
zero set in any C*-algebra is the set of matrix units of copies of $F$
and the set of matrix units of unital copies of $F$, respectively.   
\end{lemma}

\noindent{\bf Note:} We shall write $\alpha_k$ and $\alpha_k^u$ instead of $\alpha_{M_k(\bbC)}$ and $\alpha_{M_k(\bbC)}^u$, respectively. 

\begin{proof} Every finite-dimensional C*-algebra is a direct sum of full matrix algebras 
(see e.g., \cite{Dav:C*}). If $F=\bigoplus_{1\leq l\leq n} M_{k(l)}(\bbC)$ (so that 
$m=\sum_{1\leq l\leq n} k(l)^2$ we define 
$\alpha_F $ to be the following formula in variables $x_{ij}^{(l)}$, for $1\leq l\leq n$ 
and $1\leq i,j\leq k(l)$:
\[
\textstyle\max_{1\leq l\leq n} \alpha_{k(l)}(x_{ij}^{(l)}:1\leq i,j\leq k(l))
+\rho_p(\sum_{1\leq l\leq n} \sum_{1\leq i\leq k(l)} x_{ii}^{(l)})
\]
The first part of the formula assures that for each $l$ the  group $x_{ij}^{(l)}$
represents units of a $k(l)\times k(l)$ matrix, and the second part assures that the units 
of these matrices add up to a projection. Since both $\alpha_k$ and $\rho_p$ are stable, 
the stability of $\alpha_F$ follows. 

We finally  let  $\alpha_F^u$ be 
\[
\textstyle
\max(\alpha_F, \|1-\sum_{1\leq l\leq n} \sum_{1\leq i\leq k(l)} x_{ii}^{(l)}\|)
\]
as in Lemma~\ref{L.alpha}. 
\end{proof}

\subsection{An alternative axiomatization of C*-algebras}

In \cite{FaHaSh:Model2} it was proved that the C*-algebras form an elementary class
in the logic for metric structures. 
In order to make the proofs more transparent, we shall introduce an alternative representation 
of C*-algebras as metric structures. 
In \cite[\S 2.3.1]{FaHaSh:Model1}  C*-algebras are represented as one-sorted structures
with sort $U$ for the algebra itself and domains $D_n$, for $n\geq 1$, corresponding to 
$n$-balls of the algebra. 

We shall expand this structure, by adding a second sort $C$ with domains $C_n$, for $n\geq 1$. 
Sort $C$ is always interpreted as the complex numbers, and~$C_n$ is the disk of
all $z$ with $|z|\leq n$. We also add 
the distance function 
\[
d_C(a,b)=|a-b|
\]
as well as constant symbols for all elements of $\bbQ+i\bbQ$. 
In addition to the axioms given in \cite[\S 3.1]{FaHaSh:Model2}, we add the following
\begin{enumerate}
\item field axioms  
for elements of $C$, 
\item  axioms that associate multiplication by  $\lambda\in \bbQ+i\bbQ$ 
(which is a part of 
the language) to the element of $C$ corresponding to $\lambda$. 
\end{enumerate}
We denote this theory by $T_{C^*}'$.
Clearly, every model of $T_{C^*}'$ has a reduct that is a model of $T_{C^*}$ and 
every model of $T_{C^*}$ has the unique expansion to a model of $T_{C^*}'$.

In  \cite[Theorem~9.15]{BYBHU} it was proved that if 
formulas $\phi(\bar x, \bar y)$ and $\psi(\bar y)$ are stable  
then  the formula $(\inf_{\psi(\bar y)=0}) \phi(\bar x, \bar y)$ is stable. 
As pointed out above,~\cite{BYBHU} considers definability over a complete theory but the
result nevertheless applies to our context.

\begin{lemma} \label{L.dist} For every finite-dimensional C*-algebra $F$ and every $m$ 
there are $m$-ary stable formulas $\beta_{F,m}$ 
and $\beta_{F,m}^u$ such that for every C*-algebra
$A$ and $a_1,\dots, a_m$ in the unit ball of $A$ we have 
\[
\beta_{F,m}(a_1,\dots, a_m)^A=\inf_C \max_{1\leq i\leq m} \dist (a_i,C)
\]
where $C$ ranges over isomorphic copies of $F$ in $A$
and 
\[
\beta_{F,m}^u(a_1,\dots, a_m)^A=\inf_C \max_{1\leq i\leq m} \dist (a_i,C)
\]
where $C$ ranges over unital isomorphic copies of $F$ in $A$. 
\end{lemma} 

\begin{proof} 
Fix $m$. Let us first prove the case when $F=M_k(\bbC)$.  Let 
$\beta^0_{k,m}(y_1,\dots, y_m)$ be
\[
\textstyle
\inf_{\bar x}(\alpha_k(\bar x)+  
 \inf_{\bar \lambda}\max_{1\leq l\leq m} \|y_l-\sum_{1\leq i,j\leq k} \lambda_{ij}^l x_{ij}\|), 
 \]
 where $\bar x$ ranges over $k^2$-tuples $x_{ij}$, $1\leq i,j\leq k$ in 
 the unit ball of the algebra and 
 $\bar \lambda$ ranges over $l\cdot k^2$-tuples $\lambda_{ij}^l$ in $\bbD=\{z\in \bbC: |z|\leq 1\}$. 
 Clearly $\beta^0_{k,m}(\bar a)^A=0$ if and only if there is a copy of $M_k(\bbC)$ in $A$ 
 including all $a_i$, $1\leq i\leq k$. Also, since we are quantifying  over the compact set 
 $\bbD^{lk^2}$, 
 the formula $\beta^0_{k,m}$ is stable by 
\cite[Theorem~9.15]{BYBHU}. 
By \cite[Proposition~9.19]{BYBHU} there exists a formula $\beta_{M_k(\bbC),m}$ 
(henceforth denoted  $\beta_{k,m}$) as required. 

The case when $F$ is an arbitrary finite-dimensional C*-algebra is only notationally 
different (cf. the proof of Lemma~\ref{L.F} ). Finally, $\beta_{F,m}^u$ is obtained by replacing $\alpha_{F,m}$ with~$\alpha_{F,m}^u$. 
 \end{proof}

\section{Proofs of Theorem~\ref{T.UHF} and Theorem~\ref{T.AF}} 
\label{S.proofs} 

A C*-algebra $A$ is \emph{locally matricial}, or LM, if for every finite subset $\bar a$ of 
$A$ and $\e>0$ there exists a $k$ and a *-homomorphism of $M_k(\bbC)$ into $A$
such that all elements of $\bar a$ are within $\e$ of its range. 
 By a classical result of Glimm (\cite{Glimm:On}, see also \cite{Dav:C*})
in the separable case this is equivalent to $A$ being UHF. 
It was shown in \cite{FaKa:Nonseparable} 
that being LM is not equivalent to being UHF for nonseparable C*-algebras. 

One similarly defines LF algebras as the algebras in which every finite set is within $\e$ of a 
copy of a finite-dimensional C*-algebra. 
By a classical result of Bratteli (\cite{Ell:On}, see also  \cite{Dav:C*})
in the separable case this is equivalent to $A$  being AF.

By Glimm's and Bratteli's results, Theorem~\ref{T.LM} and Theorem~\ref{T.LF}  imply
 Theorem~\ref{T.UHF}
and Theorem~\ref{T.AF}, respectively. 

\begin{thm}\label{T.LM} There are countably many types 
 such that 
a separable C*-algebra $A$ is LM if and only if it omits all of these types. 
\end{thm}

\begin{proof}
We shall  produce a countable family of types $\bt_{m,n}$, for $m,n\in \bbN$, so that 
$\bt_{m,n}$ is an $n$-type
and a separable C*-algebra $A$ is UHF if and only if it omits all of these types.

For  $k$ and $n$ in $\bbN$ let $\beta_{k,n}$ denote the formula $\beta_{M_k(\bbC),n}$ (as 
in Lemma~\ref{L.dist}). 
Let $\bt_{m,n}$ be the $n$-type in $x_1,\dots, x_n$ consisting of all the conditions  
$\beta_{k,n}(\bar x)\geq 1/m$ for $k\geq 2$ and all the conditions $\|x_i\|\leq 1$, for $i\leq n$.

Type $\bt_{m,n}$ is realized by $a_1,\dots, a_n$ in algebra $A$ if and only if
each $a_i$ belongs to the unit ball and 
 every subalgebra $C$ of $A$ that is isomorphic to a full matrix algebra
 is such that $\max_{1\leq i\leq n} \dist(a_i,C)\geq 1/m$. 
\end{proof}

\begin{thm}\label{T.LF} 
There are countably many types 
 such that 
a separable C*-algebra $A$ is LF if and only if it omits all of these types. 
\end{thm}

\begin{proof}
This proof is analogous to the above proof of Theorem~\ref{T.UHF}. 
 We shall  produce a countable family of types $\bs_{m,n}$, for $m,n\in \bbN$, so that 
$\bs_{m,n}$ is an $n$-type
and a separable C*-algebra $A$ is AF if and only if it omits all of these types.

Let $\bs_{m,n}$ be the $n$-type in $x_1,\dots, x_n$ consisting of all the conditions  
$\beta_{F,n}(\bar x)\geq 1/m$, where $F$ ranges over all finite-dimensional 
C*-algebras and all the conditions $\|x_i\|\leq 1$, for $i\leq n$. 

Type $\bs_{m,n}$ is realized by $a_1,\dots, a_n$ in algebra $A$ if and only if
each $a_i$ belongs to the unit ball and 
 every subalgebra $C$ of $A$ that is isomorphic to a finite-dimensional C*-algebra
 is such that $\max_{1\leq i\leq n} \dist(a_i,C)\geq 1/m$. 
 \end{proof} 

\begin{proof}[Proof  of Theorem~\ref{T.3}]
(1) We prove that if  $A$ and $B$ are unital, separable UHF algebras then they are isomorphic if and only if they are elementarily equivalent. 
By Glimm's theorem (\cite{Glimm:On}) a complete isomorphism invariant for unital UHF algebras
is given by the \emph{generalized integer} $\bfn$ defined as the formal product
\[
\textstyle \bfn(A)=\prod_{p\text{ prime}} p^{n(A)}
\]
where $n(A)$ is the supremum of all $n$ such that $A$ includes a unital copy of~$M_{p^n}(\bbC)$. 
 
By Lemma~\ref{L.alpha}, $A$ has a unital copy of $M_k(\bbC)$ if and only 
if $\inf_{\bar x} \alpha_k^u(\bar x)^A=0$. 
Therefore $\bfn(A)$ can be recovered from the theory of $A$, and 
the conclusion follows by Glimm's theorem.

(2) We  give a non-constructive proof that  there are non-unital, separable UHF algebras which are elementarily equivalent but not isomorphic. 
Essentially by Dixmier's classification result for these algebras (\cite{Dix:Some}), 
to every countable, torsion free, rank one abelian group $\Gamma$ one can associate a
separable non-unital UHF algebra $A(\Gamma)$ so that 
$A(\Gamma)\cong A(\Gamma')$ 
if and only if $\Gamma\cong\Gamma'$. 
The group $\Gamma$ is $K_0(A(\Gamma))$ and
the map $\Gamma\mapsto A(\Gamma)$ is implemented by a Borel map from the 
standard Borel space of countable groups into the standard Borel space of 
separable C*-algebras (\cite{FaToTo:Descriptive}).
Since the isomorphism of 
countable, torsion free, rank one abelian groups is not a smooth Borel equivalence relation
(see e.g., \cite{Tho:On}), we conclude that the isomorphism of 
non-unital, separable, UHF algebras is not smooth either.  
Finally, the computation of the theory of a separable C*-algebra is given by a Borel map (\cite{FaToTo:Descriptive}) and the theory (being a smooth invariant) 
therefore cannot be a complete invariant for non-unital, separable,  UHF algebras.

(3) In order to see that  there are unital separable AF algebras which are elementarily equivalent
but not isomorphic, consider unitizations of algebras constructed in (2). 
\end{proof}

\section{Prime models} 

A model of a theory $T$ is \emph{prime} if it is isomorphic to an elementary submodel of
every other model of $T$.  

\begin{thm} \label{T.atomic} Every separable unital UHF algebra is a prime model. 
\end{thm}

The proof depends on a sequence of lemmas. $\bt_A(\bar a)$ denotes the type of $\bar a$ in $A$ and `type' means `type over the empty set' unless otherwise specified.  Let us first record three obvious facts

\begin{lemma}\label{L.1} 
\begin{enumerate}
\item Suppose that $\varphi$ is a stable formula and $A$ is a model such that for $\bar a,\bar b \in A$, if $\varphi(\bar a) = \varphi(\bar b) = 0$ then $\bt_A(\bar a) = \bt_A(\bar b)$.  
Then $\varphi$ determines a principal type.
\item If $\varphi$ is a stable formula and $A$ is a model such that whenever $\varphi(\bar a) = \varphi(\bar b) = 0$ then there is an automorphism $\sigma$ of $A$ such that $\sigma(\bar a) = \bar b$ then $\bt_A(\bar a) = \bt_A(\bar b)$ and this type is principal.
\item If $A$ is a model such that for all $a \in A$ and $\epsilon > 0$ there is $\bar b \in A$ such that 
$\bt_A(\bar b)$ is principal and $d(a,\bar b) \leq \epsilon$ then $A$ is atomic. 
\end{enumerate}
\end{lemma}

\begin{proof} (1) and (2) are obvious. The proof of (3) is similar to the proof of 
 \cite[Corollary~12.9]{BYBHU}. 
 One chooses a countable dense subset of $A$ such that every finite subset has a principal 
type over the empty set. 
A straightforward `forth' (half of back-and-forth) argument shows that 
$A$ is isomorphic to a unital  subalgebra of any other model of the theory of~$A$. 
\end{proof}

We consider the finite-dimensional case first. 

\begin{lemma} \label{L.atomic} 
Assume $F$ is a  finite-dimensional C*-algebra
and $\bar b$ is an $m$-tuple in $F_1$.  
Then there exists a stable $m$-ary  formula $\gamma_{F,\bar b}(\bar x)$
such that in a  unital C*-algebra $A$ 
the zero-set of $\gamma_{F,\bar b}$ is equal to 
the set of all $\bar  c\in (A_1)^m$ such that there exists 
a unital *-isomorphism $\Phi\colon F\to A$ such that $\Phi(\bar b)=\bar c$. 
\end{lemma} 

\begin{proof} Let us first consider the case when $F$ is $M_k(\bbC)$. 
Fix for a moment matrix units $a_{ij}$, for $1\leq i,j\leq k$ of $F$. 
Any $m$-tuple of elements $b_1,\dots, b_m$ in $F$ has coordinates $\lambda_{ij}^l$, 
for $1\leq l\leq m$ and $1\leq i,j\leq k$ such that $b_l=\sum_{1\leq i,j\leq k} \lambda_{ij}^l a_{ij}$. 
Recall that the complex numbers are part of the language and 
let  $\gamma_{F,\bar b}(\bar y)$ be 
\[
\textstyle \inf_{\bar x} \alpha_k^u(\bar x)+\sum_{1\leq l\leq m} \|y_l-
\sum_{1\leq i,j\leq k} \lambda_{ij}^l x_{ij}\|.
\]
Then $\gamma(\bar b)=0$. 
The stability of $\alpha_k$ implies that $\gamma_{F,\bar b}$ is stable.
If $\gamma(\bar c)=0$ for $\bar c \in A$ then 
there is  $k^2$-tuple $\bar d$ such that $\alpha_k^u (\bar d)=0$ 
and $c_l=\sum_{1\leq i,j\leq k} \lambda_{ij}^l d_{ij}$. 
The *-homomorphism from $M_k(\bbC)$ to $A$ that sends 
$a_{ij}$ to $d_{ij}$ also sends $\bar b$ to $\bar c$.
and it is clear that the zero-set of $\gamma_{F,\bar b}$ is as required. 

The case of general finite-dimensional algebra $F$ is almost identical.
\end{proof}

\begin{lemma} \label{L.elem} 
Assume $A$ is a unital separable UHF algebra.  
 Then every unital embedding of 
$A$ into an elementarily equivalent C*-algebra $B$ is
automatically elementary. 
\end{lemma} 

\begin{proof} By the L\"owenheim--Skolem theorem it suffices to prove the result for separable $B$. 
Fix a nonprincipal ultrafilter $\cU$ on $\bbN$. 
Since  $B$ can be elementarily embedded into the ultrapower $A^{\cU}$ by countable saturation, 
it will suffice to show that  every unital embedding of $A$ into $A^{\cU}$ is elementary. 

Since any two unital copies of $M_k(\bbN)$ in $A$ are conjugate  (e.g., see~\cite{Dav:C*}), 
and since this can be expressed in our logic, by \L os's theorem 
any two unital copies of $M_k(\bbN)$ in $A^{\cU}$ are conjugate. 
By countable saturation of $A^{\cU}$ and separability of $A$ we conclude that any two 
unital copies of $A$ in $A^{\cU}$ are conjugate. Since the diagonal map is an elementary 
embedding of $A$ into $A^{\cU}$, every other embedding is elementary as well. 
\end{proof}

\begin{proof}[Proof of Theorem~\ref{T.atomic}]
Recall that a  model $A$ of a theory $T$ is \emph{atomic} if every type 
over the empty set
realized in $A$ is principal. It is well-known that a separable model is prime if and only if 
it is atomic (see e.g., \cite[Corollary~12.9]{BYBHU} for the nontrivial direction). 
We shall use a minor strengthening of this fact. 

Fix a unital UHF algebra $A$.  
We first show that for every $k$ and every unital copy $C$ 
of $M_k(\bbC)$ in $A$, the type over the empty set 
of every $m$-tuple~$\bar a$ of elements of $C$ 
is principal. 

Consider  the stable formula $\gamma=\gamma_{k,\bar a}$ 
used in the proof of Lemma~\ref{L.atomic}.  
Then $\gamma(\bar a)^A=0$. 
Moreover, any other $m$-tuple~$\bar b$ in $A$ we have
that  $\gamma(\bar b)^A=0$ if and only if $\bar b$ 
is contained in a unital copy of $M_k(\bbC)$ in $A$ with the same coefficients as $\bar a$. 
Since for every $k$ that divides the generalized integer of $A$, 
any two unital copies of $M_k(\bbC)$ in $A$ are conjugate (see e.g., \cite{Dav:C*}), 
 Lemma~\ref{L.1} (2) implies that $\bar a$ and $\bar b$ have the same type. 

We conclude that  the type of $\bar a$ is determined by a single stable formula, $\gamma(\bar x)$
and therefore principal. 


Since $A$ is UHF, 
every $m$-tuple in $A$ can be approximated arbitrarily well  by an $m$-tuple 
belonging to a unital copy of $M_k(\bbC)$ for a large enough $k$.
We  have therefore proved that for every $m$ the set of 
$m$-tuples whose type is principal is dense in $A^m$, and by Lemma~\ref{L.1} (3) 
we conclude that $A$ is atomic. 
\end{proof} 

A salient point of the above proof of Theorem~\ref{T.atomic} is that 
all unital copies of $M_n(\bbC)$ in a UHF algebra are conjugate. 
We should note that this is not necessarily the case for an arbitrary  simple, 
nuclear, separable, unital C*-algebra. It is an easy consequence of the Kirchberg--Phillips classification 
theorem for purely infinite C*-algebras (see e.g., \cite{Ror:Classification}) that there exists a nuclear, simple, separable,
unital C*-algebra $A$ and two unital copies of~$M_2(\bbC)$ in $A$ that are not conjugate
(see e.g.,~\cite{FaKa:NonseparableII}). 

Our proof of Theorem~\ref{T.atomic} actually shows the following. 

\begin{prop} \label{P.atomic} Assume $A$ is a separable, unital C*-algebra which is a
direct limit of algebras $A_n$, for $n\in \bbN$, so that (i) each $A_n$ is 
finitely generated, and the generators are defined by  a stable quantifier-free formula, 
and (ii) every two unital copies of $A_n$ in $A$ are conjugate. 
Then $A$ is atomic. 
\end{prop}

\begin{proof} The proof closely follows the proof of Theorem~\ref{T.atomic}. 
First, every $\bar a$ in the subalgebra of $A_n$ algebraically generated by a 
fixed set of generators has a principal type. Since the generators are given by 
a stable, quantifier-free formula, the formula defining this type is true in $A$
(here $A_n$ is considered as a subalgebra of $A$). But any two copies of $A_n$ in $A$ 
are conjugate, and therefore this formula completely determines the type of $\bar a$ in $A$. 
We  have therefore proved that for every $m$ the set of 
$m$-tuples whose type is principal is dense in $A^m$, and by Lemma~\ref{L.1} (3) 
we conclude that $A$ is atomic. 
\end{proof} 
 
 The Jiang--Su algebra $\cZ$ has a special place in Elliott's program (\cite{JiangSu}, see also \cite{Win:Ten}). 
  It is a direct limit of \emph{dimension drop algebras $\cZ_{p,q}$}, which are 
  finitely-generated algebras whose generators 
 are defined by a stable quantifier-free formula (\cite[\S 7.2]{JiangSu}). 
However, Proposition~\ref{P.atomic} does not apply to $\cZ$ since every dimension-drop algebra $\cZ_{p,q}$
has many non-conjugate unital copies in $\cZ$. The reason for this is that $\cZ$ has a unique trace, 
while $\cZ_{p,q}$ does not, hence the conjugacy orbit of a copy of $\cZ_{p,q}$ depends on 
which one of its traces extends to a trace in $\cZ$, and it is not difficult to see that this could
be any of the traces of $\cZ_{p,q}$. 
Results of Matui and Sato (\cite[Lemma~4.7]{MaSa:Z-stabilityII}) on the uniqueness of trace
suggest that the unique trace is definable in 
 nuclear C*-algebras with property (SI) (in particular, in $\cZ$).

Since two  separable atomic models are isomorphic iff they are  elementarily equivalent, 
Theorem~\ref{T.atomic} provides an alternative proof of Theorem~\ref{T.3} (1). 
In the other direction, parts (2) and (3) of Theorem~\ref{T.3} immediately imply the following. 

\begin{prop} \label{P.nonatomic} 
Some non-unital UHF algebras are not atomic models. 
Some separable AF algebras are not atomic models. \qed
\end{prop} 

By Elliott's classification of separable AF algebras by the ordered $K_0$
 (\cite{Ell:On}, \cite{Ror:Classification}), the isomorphism type of a separable 
 AF algebra $A$ is determined by $(K_0(A),K_0(A)^+,[1])$. 
 In particular, the complete information on what non-principal types are being realized in 
 $A$ is contained in the ordered $K_0$.
A better model-theoretic 
understanding of the mechanism behind this  would 
be desirable. 
In particular, what is the connection between $K_0(A)$
and the theory of $A$, and what groups correspond to elementarily equivalent algebras? 
How exactly does $K_0(A)$ determine what nonprincipal types are realized in $A$?

\section{Concluding remarks} 

As pointed out in the introduction, we believe that model-theoretic study of 
nuclear C*-algebras, initiated in the present paper,  will be fruitful. 

We conclude by stating some of the many questions along this line of research. 
Does the theory of every C*-algebra allow an atomic model? 
Are all atomic models of the theory of C*-algebras nuclear? 
(The converse is, by Theorem~\ref{T.3}, false.) Is there a nuclear C*-algebra elementarily equivalent 
to the  Calkin algebra?
Is there a characterization of Elliott invariants of C*-algebras that are atomic models? 
Is every C*-algebra elementarily equivalent to a nuclear C*-algebra? 
  A positive answer to this question would imply a positive answer to a problem 
  of Kirchberg, whether every separable C*-algebra is a subalgebra of the ultrapower of 
  Cuntz algebra $\cO_2$. Of course Kirchberg's problem (a C*-algebraic version of Connes Embedding Problem) is asking whether the universal theory of every C*-algebra includes the universal theory of $\cO_2$.

\bibliographystyle{amsplain}
\bibliography{ifmain}

\end{document}